\documentclass[11pt]{article}
\usepackage{amsmath,amsthm,amsfonts,amssymb,url,graphicx,xcolor,authblk}
\newtheorem{theorem}{Theorem}[section]

\newcommand{\B}{\ensuremath{\mathcal{B}}} 
 
\newcommand{\zed}{\ensuremath{\mathbb{Z}}} 
\newcommand{\fb}[1]{\framebox(26,12){#1}} 
\newcommand{\fbb}[1]{\framebox(36,12){#1}} 
\definecolor{red}{RGB}{220,20,60}
\newcommand{\RED}[1]{\textcolor{red}{#1}}
\newcommand{\third}{\ensuremath{\mathsf{third}}}

\title{Block-avoiding sequencings of points in Steiner triple systems}
\author{Donald L.\ Kreher}
\affil{Department of Mathematical Sciences,
Michigan Technological University,
Houghton, Michigan, 49931,
U.S.A.}
\author{Douglas R.\ Stinson%
\thanks{D.R.\ Stinson's research is supported by  NSERC discovery grant RGPIN-03882.}}
\affil{David R.\ Cheriton School of Computer Science, University of Waterloo,
Waterloo, Ontario, N2L 3G1, Canada}
\date{\today}

\begin{document}
\maketitle

\begin{abstract}
Given an STS$(v)$, we ask if there is a permutation of the points of the design such that no $\ell$ consecutive points in this permutation contain a block of the design. Results are obtained in the cases $\ell = 3,4$.
\end{abstract}

\section{Introduction}

A \emph{Steiner triple system of order $v$} is a pair $(X, \B)$, where 
$X$ is a set of $v$ \emph{points} and $\B$ is a set of 3-subsets of $X$ (called 
\emph{blocks}), such that every pair of points occur in exactly one block.
We will abbreviate the phrase ``Steiner triple system of order $v$'' to 
STS$(v)$.

It is well-known that an STS$(v)$ contains exactly $v(v-1)/6$ blocks, and an 
STS$(v)$ exists if and only  if $v \equiv 1,3 \bmod 6$. The definitive reference for Steiner triple systems is the book \cite{CR} by Colbourn and Rosa.

Suppose $(X, \B)$ is an STS$(v)$. We ask if there is a permutation 
(or \emph{sequencing}) of the points in $X$ so that no 
three consecutive points in the sequencing 
comprise a block in $\B$. That is, can we fid a sequencing 
$\pi = [x_1\; x_2 \;  \cdots \;  x_v]$ of $X$ such that $\{ x_i, x_{i+1}, x_{i+2}\} \not\in \B$ for all $i$, $1 \leq i \leq v-2$? Such a sequencing will be termed a \emph{$3$-good} sequencing for the given STS$(v)$.

More generally, we could ask if there is a sequencing of the points such that no $\ell$ consecutive points in the sequencing contain a block in $\B$. Such a sequencing will be termed $\ell$-good for the given STS$(v)$. 

As an example, consider the STS$(7)$ $(X,\B)$, where $X = \zed_7$ and $\B = 
\{ 013, 124, 235, 346, 450, 451, 562\}$. The sequencing 
$[0\; 1 \;  2 \; 3\; 4 \;  5 \; 6]$ is easily seen to be $3$-good. However, it is not 
$4$-good, as the block $013$ is contained in the first four points of the sequencing.
(Note that, here and elsewhere, we might write blocks $\{x,y,z\}$ as $xyz$ if the context is clear.)

A \emph{partial Steiner triple system of order $v$} is a pair $(X, \B)$, where 
$X$ is a set of $v$ \emph{points} and $\B$ is a set of 3-subsets of $X$ (called 
\emph{blocks}), such that every pair of points occur in at most one block.
We will abbreviate the phrase ``partial Steiner triple system of order $v$'' to 
partial STS$(v)$ or PSTS$(v)$.  There are no congruential restrictions on the values $v$ for which PSTS$(v)$ exist. We will also consider $\ell$-good sequencings of 
PSTS$(v)$.

The main results we prove in this paper are that every STS$(v)$ with $v>3$ has a $3$-good sequencing, and every STS$(v)$ with $v > 71$ has a 4-good sequencing. Similar results are obtained for PSTS$(v)$ as well.

We will use the following notation. Suppose $(X,\B)$ is an STS$(v)$.
Then, for any pair of points $x,y$, 
let $\third(x,y) = z$ if and only if $\{x,y,z\} \in \B$.
The function $\third$ is well-defined because every pair of points occurs in a 
unique block in $\B$.

\subsection{Background and motivation}

Brian Alspach gave a talk entitled ``Strongly Sequenceable Groups'' 
at the 2018 Kliakhandler Conference, which was held at Michigan Technological University. 
In this talk, among other things, the notion of sequencing diffuse posets was introduced
and the following research problem was posed:

\begin{quote}
``Given a triple system of order $n$ with $\lambda = 1$, define a poset $P$ by
letting its elements be the triples and any union of disjoint triples.
This poset is not diffuse in general, but it is certainly possible that
$P$ is sequenceable.''
\end{quote}

A \emph{sequenceable} STS$(v)$ (or PSTS$(v)$
is an STS$(v)$ in which the points can be ordered 
(i.e., sequenced)  so that no  $t$ consecutive points can be partitioned into $t/3$ blocks, for any $t \equiv 0 \bmod 3$, $t < v$. The problem is studied in Alspach, Kreher and Pastine \cite{AKP}.

One possible relaxation of the definition of sequenceable STS$(v)$ would be to require 
a sequencing of the points so that no $t$ consecutive points can be partitioned into $t/3$ blocks, for any $t \equiv 0 \bmod 3$, $t \leq w$, where $w < v$ is some specified integer. Such an  STS$(v)$ could be termed \emph{$w$-semi-sequenceable}.

A $3$-semi-sequenceable STS$(v)$ has a sequencing of the points so that no three consecutive points form a block. This is identical to a ``$3$-good sequencing.''
As noted above, we then generalize this notion to $\ell$-good sequencings and 
we consider the case $\ell = 4$ in detail.  

Although we do not explicitly study $w$-semi-sequenceable
STS in this paper, we note the following connection between $w$-semi-sequenceable STS$(v)$ and
STS$(v)$ having $\ell$-good sequencings.

\begin{theorem} An STS$(v)$ that has a $(2u+1)$-good sequencing is $3u$-semi-sequenceable.
\end{theorem}

\begin{proof} Let $\pi$ be a sequencing of the points of an STS$(v)$.
Suppose $t \equiv 0 \bmod 3$ and suppose there are $t$ consecutive points
in $\pi$ that can be partitioned into $t/3$ blocks of the STS$(v)$. Let these $t$ points be denoted 
(in order) $x_1, \dots , x_t$. Then 
\[ \{x_1, \dots , x_t\} = \bigcup _{j=1}^{t/3} B_j,\]
where $B_1, \dots , B_{t/3}$ are blocks in the STS$(v)$. For $1 \leq j \leq t/3$, let
\[m_{lo}(j) = \min \{ i : x_i \in B_j\}\] and let \[m_{hi}(j) = \max \{ i : x_i \in B_j\}.\]
Clearly there is a block $B_j$ such that $m_{lo}(j) \geq t/3$. It also holds that 
$m_{hi}(j) \leq t$. Therefore the block $B_j \subseteq \{x_{t/3}, \dots , x_t\}$, which means that
the sequencing $\pi$ is not $(2t/3 +1)$-good.
\end{proof}


\section{Existence of $3$-good sequencings} 

In this section, we show that there is a $3$-good sequencing for any STS$(v)$ with $v>3$, as well as for any PSTS$(v)$ with $v  > 3$. We prove these facts in two ways: first, by a counting argument, and second, by using a greedy algorithm.

\subsection{A counting argument}

Let $(X, \B)$ be an STS$(v)$ on points $X = \{1, \dots , v\}$.
For a sequencing $\pi = [x_1\; x_2 \;  \cdots \;  x_v]$ of $X$, 
and for any $i$, $1 \leq i \leq v-2$,
define $\pi$ to be \emph{$i$-forbidden} if $\{x_i, x_{i+1}, x_{i+2}\} \in \B$. 
Let $\mathsf{forbidden}(i)$ denote the set of $i$-forbidden sequencings.
Also, define a sequencing to be \emph{forbidden} if it is $i$-forbidden for at least one value
of $i$ and let $\mathsf{forbidden}$ denote the set of forbidden sequencings.
Clearly, a sequencing is $3$-good if and only if it is not forbidden.

\begin{theorem}
\label{T1}
Suppose $v > 3$ and $(X, \B)$ is an STS$(v)$ on points $X = \{1, \dots , v\}$.
Then there is a sequencing $\pi = [x_1\; x_2 \;  \cdots \;  x_v]$ of $X$ that is $3$-good
for $(X, \B)$.
\end{theorem}

\begin{proof} 
Clearly, \[\mathsf{forbidden} = \bigcup _{i=1}^{v-2} \mathsf{forbidden}(i). \]
For any given value of $i$, it holds that 
$|\mathsf{forbidden}(i)| = v! / (v-2)$. This follows because, for any two  points,
 $x_i$ and $x_{i+1}$, the 3-subset $\{x_i, x_{i+1}, x_{i+2}\} \in \B$
if and only if $x_{i+2} = \third(x_i, x_{i+1})$.
So given any $x_i$ and $x_{i+1}$, the probability that $\{x_i, x_{i+1}, x_{i+2}\} \in \B$
is $1/(v-2)$.

Next, by the union bound, 
\begin{equation}
\label{count.eq}
 |\mathsf{forbidden}| \leq \sum_{i=1}^{v-2} |\mathsf{forbidden}(i)| = (v-2) \times \frac{v!}{ (v-2)} = v!
 \end{equation} 
Equality in (\ref{count.eq}) would be obtained if and only if the sets $\mathsf{forbidden}(i)$, $1 \leq i \leq v-2$, are pairwise disjoint. 

We show that equality in (\ref{count.eq}) is impossible: Consider any two intersecting blocks
$\{a,b,c\}, \{c,d,e\} \in \B$ (here is where we use the assumption that $v > 3$). 
Then any sequencing in which the first five symbols are $a\; b\; c\; d\; e$ (in that order) is in
$\mathsf{forbidden}(1) \cap \mathsf{forbidden}(3)$. Therefore,
$|\mathsf{forbidden}| < v!$ and thus there exists a $3$-good sequencing.
\end{proof}

Theorem \ref{T1} also holds for  \emph{partial} STS$(v)$ when $v > 3$.
\begin{theorem}
\label{T2}
Suppose $v > 3$ and $(X, \B)$ is a partial STS$(v)$ on points $X = \{1, \dots , v\}$.
Then there is a sequencing $\pi = [x_1\; x_2 \;  \cdots \;  x_v]$ of $X$ that is $3$-good
for $(X, \B)$.
\end{theorem}

\begin{proof} 
If $(X, \B)$ is an STS$(v)$, then we are done by Theorem \ref{T1}. Therefore, we can assume there is 
at least one pair $\{a,b\}$ that does not appear in any block in $\B$.
Suppose $x_i =a$ and $x_{i+1} = b$. Then, for every possible $x_{i+2}$, we have 
$\{x_i, x_{i+1}, x_{i+2}\} \not\in \B$. It then follows that $|\mathsf{forbidden}(i)| < v! / (v-2)$ for all $i$.

Now, when we apply the union bound, we have
\[ |\mathsf{forbidden}| \leq \sum_{i=1}^{v-2} |\mathsf{forbidden}(i)| < (v-2) \times v! / (v-2) = v!\] 
and we are done.
\end{proof}

\subsection{A greedy algorithm}

Theorems \ref{T1} and \ref{T2} can also be proven using a greedy algorithm.
First, we consider the case where $(X,\B)$ is an STS$(v)$.
Suppose we begin by choosing any two distinct values for $x_1$ and $x_2$.
Now, consider any  $i$ such that $3 \leq i \leq v-1$. 
Clearly we must have $x_i \not\in \{x_1, \dots , x_{i-1}\}$. Also, 
$x_i \neq \third(x_{i-2},x_{i-1})$. So there are at most
$i$ values for $x_i$ that are ruled out. Since $i \leq v-1$, there is at least
one value for $x_i$ that does not violate the required conditions. 

After choosing $x_1, x_2, \dots ,x_{v-1}$ as described above, there is only one
unused value remaining for $x_v$. But this might not result in a $3$-good sequencing, 
if it happens that $\{x_{v-2}, x_{v-1},x_v\} \in \B$.
However, in this case, it turns out that we can find a slight modification of of the sequencing 
$[x_1\; x_2 \;  \cdots \;  x_v]$ that is $3$-good,
provided that $v > 5$. 

Suppose we made sure to select $x_5$ such that $\{x_2,x_3,x_5\} \in \B$, i.e., we define  
$x_5 = \third (x_2,x_3)$.
This is an allowable choice for $x_5$ because 
\begin{itemize}
\item $\{x_1,x_2,x_3\} \not\in \B$ and
$\{x_2,x_3,x_4\} \not\in \B$, which implies that \[x_5 \not\in \{ x_1, x_2,x_3,x_4\},\] and
\item $\{x_3,x_4,x_5\} \not\in \B$, because $\{x_2,x_3,x_5\} \in \B$ and $x_2 \neq x_4$.
\end{itemize}
Now, suppose we have a sequencing $[x_1\; x_2 \; \cdots \; x_v]$, where $\{x_2,x_3,x_5\} \in \B$, 
which fails to be $3$-good only because $\{x_{v-2}, x_{v-1},x_v\} \in \B$ (which is not allowed). Consider the
modified sequencing $[y_1\; y_2 \; \cdots \; y_v]$ obtained from $[x_1\; x_2 \; \cdots \; x_v]$ by
switching $x_1$ and $x_v$. In order to show that $[y_1\; y_2 \; \cdots \; y_v]$ is
a $3$-good sequencing, we need to show that
\begin{enumerate}
\item $\{y_{v-2},y_{v-1},y_v\} = \{x_{v-2},x_{v-1},x_1\} \not\in \B$, and
\item $\{y_1,y_2,y_3\} = \{x_v,x_2,x_3\} \not\in \B$.
\end{enumerate}
To prove 1, we observe that $\{x_{v-2},x_{v-1},x_1\} \not\in \B$ because 
$\{x_{v-2},x_{v-1},x_v\} \in \B$ and $x_v \neq x_1$. To prove 2,
we observe that $\{x_2,x_3,x_5\} \in \B$ and $x_v \neq x_5$ because $v > 5$.
Thus the sequencing $[y_1\; y_2\; \cdots \;y_v]$ is $3$-good.

The above-described process can also be carried out to find a $3$-good sequencing for any partial 
STS$(v)$ with $v > 5$. The resulting algorithm is presented in Figure \ref{alg1}.

\begin{figure}[htb]
\begin{enumerate}
\item Choose a block $\{b,c,e\} \in \B$, let $a \neq b,c,e$
and let $d \neq a,b,c,e$.
\item Define $x_1 = a$, $x_2 = b$, $x_3 = c$, $x_4 = d$ and $x_5 = e$.
\item {\bf For} $i = 6$ {\bf to} $v-1$ {\bf do} define $x_i$ to be any 
element of $X$ that is distinct from the values $x_1, \dots , x_{i-1}$ and 
$\third(x_{i-2}, x_{i-1})$. 
\item Define $x_v$ to be the unique value that is distinct from $x_1, \dots , x_{v-1}$.
\item {\bf If} $\{ x_{v-2}, x_{v-1}, x_v\} \in \B$ {\bf then} interchange $x_1$ and $x_v$.
\item {\bf Return} $(\pi = [x_1\; x_2 \; \cdots \; x_v])$.
\end{enumerate}
\caption{Algorithm to find a $3$-good sequencing for a partial STS$(v)$, $(X,\B)$}
\label{alg1}
\end{figure}

From the discussion above, we have the following theorem.

\begin{theorem}
Suppose that $(X,\B)$ is a partial STS$(v)$ with $v >5$. Then the Algorithm 
presented in Figure \ref{alg1} will find a sequencing $\pi$ that is $3$-good for
$(X,\B)$.
\end{theorem}

\section{4-good sequencings}

It is tempting to conjecture that, for any $\ell$, all
``sufficiently large'' STS have $\ell$-good sequencings. 
In this section, we prove this conjecture for the case $\ell = 4$.

We might attempt to construct a $4$-good sequencing by a greedy approach
similar to that used in the Algorithm presented in Figure \ref{alg1}.
In general, when we choose a value for $x_i$, it must be distinct from 
$x_1, \dots , x_{i-1}$, of course. It is also required that
\[ x_i \neq \third(x_{i-3},x_{i-2}), \third(x_{i-3},x_{i-1}) \text{ or } \third(x_{i-2},x_{i-1}).\]

There will be a permissible choice for $x_i$ provided that
$i-1 + 3 \leq v-1$, which is equivalent to the condition $i \leq v-3$.
Thus we can define $x_1, x_2, \dots , x_{v-3}$ in such a way that they satisfy the relevant 
conditions, and our task would be to somehow fill in the last three positions of
the sequencing, after appropriate modifications, to satisfy the desired properties.
We describe how to do this now, for sufficiently large values of $v$.


Now, suppose that $[x_1\; x_2\; \cdots \; x_{v-3}]$ is a $4$-good partial sequencing of
$X = \{1, \dots ,v\}$. Let 
$\{\alpha_1,\alpha_2,\alpha_3\} = X \setminus \{ x_1, x_2, \dots , x_{v-3}\}$.
Also, let 
\[
\begin{array}{l} \beta_1 = \third(x_{v-5},x_{v-4}),\\
 \beta_2 = \third(x_{v-5},x_{v-3}), \text{ and }\\
\beta_3 = \third(x_{v-4},x_{v-3}).
\end{array}\]
Clearly $\beta_1,\beta_2$ and $\beta_3$ are distinct.
Observe that $x_{v-2}$ and $x_{v-1}$ must be chosen so that $x_{v-2} \neq \beta_1,\beta_2, \beta_3$
and $x_{v-1} \neq \beta_3$.

By  permuting $\alpha_1,\alpha_2,\alpha_3$ if necessary,
we can assume the following two conditions hold:
\begin{equation}
\label{alpha2.eq}
\alpha_2 \neq \beta_3
\end{equation}
and
\begin{equation}
\label{xv-3.eq}
x_{v-3} \neq \third(\alpha_2,\alpha_3).
\end{equation}

Now, define the following:
\[ 
\begin{array}{l} \gamma = \third(\alpha_2,x_{v-3}),\\
\delta = \third(\alpha_2,x_{v-4}),\\
\epsilon = \third(\alpha_3,x_{v-3}),
 \text{and} \\ \eta = \third(\alpha_2, \alpha_3).
 \end{array}\]
%
Next, suppose we define $x_{v-2} = \chi$, $x_{v-1} = \alpha_2$ and  $x_v = \alpha_3$,
where 
\begin{equation}
\label{t.eq}
\chi \not\in \{ x_{v-5},x_{v-4},x_{v-3},\beta_1,\beta_2,\beta_3,\gamma,\delta,\epsilon, \eta\}
\end{equation}
is to be determined.
Thus,  the last six elements of the sequencing will be
\[ x_{v-5} \; x_{v-4} \; x_{v-3}\; \chi \; \alpha_2 \; \alpha_3.\]

There should be no block in $\B$ contained in any four consecutive points chosen
from these six points. We enumerate all the triples and verify that none of them
are blocks:
\begin{center}
\begin{tabular}{c|c}
triple & explanation \\ \hline
$\{x_{v-5} , x_{v-4} , x_{v-3}\}$ & greedy algorithm ensures it is not a block\\
$\{x_{v-5} , x_{v-4} , \chi\}$ & $\{x_{v-5} , x_{v-4} , \beta_1\}$ is a block and $\chi \neq \beta_1$\\
$\{x_{v-5} , x_{v-3} , \chi\}$ & $\{x_{v-5} , x_{v-3} , \beta_2\}$ is a block and $\chi \neq \beta_2$\\
$\{x_{v-4} , x_{v-3} , \chi\}$ & $\{x_{v-4} , x_{v-3} , \beta_3\}$ is a block and $\chi \neq \beta_3$\\
$\{x_{v-4} , x_{v-3} , \alpha_2\}$ & $\{x_{v-4} , x_{v-3} , \beta_3\}$ is a block and $\alpha_2 \neq \beta_3$ by (\ref{alpha2.eq})\\
$\{x_{v-4} , \chi , \alpha_2\}$ & $\{x_{v-4} , \delta , \alpha_2\}$ is a block and $\chi \neq \delta$\\
$\{x_{v-3} , \chi , \alpha_2\}$ & $\{x_{v-3} , \gamma , \alpha_2\}$ is a block and $\chi \neq \gamma$\\
$\{x_{v-3} , \chi , \alpha_3\}$ & $\{x_{v-3} , \epsilon , \alpha_3\}$ is a block and $\chi \neq \epsilon$\\
$\{x_{v-3} , \alpha_2 , \alpha_3\}$ & this is not a block by (\ref{xv-3.eq})\\
$\{\chi , \alpha_2 , \alpha_3\}$ & $\{\eta , \alpha_2 , \alpha_3\}$ is a block and $\chi \neq \eta$.
\end{tabular}
\end{center}
Suppose $v \geq 14$.
Our strategy is to define $\chi$ to be one of $x_1, x_2$,  $\dots$, $x_8$, in such a way that
(\ref{t.eq}) is satisfied. Note that $v - 5 \geq 9$ so we are guaranteed that 
$\chi \neq x_{v-5} ,x_{v-4} , x_{v-3}.$
We can choose $\chi\in \{x_1, x_2 , \dots ,  x_8\}$ because at least one of these eight values
is not in the set $\{\beta_1,\beta_2,\beta_3,\gamma,\delta, \epsilon,\eta\}$, which has size 7. 
Suppose we take $\chi = x_{\kappa}$, where $\kappa \in \{1,2, \dots , 8\}$.
Then we redefine $x_{\kappa} = \alpha_1$. Another way to describe this process is to temporarily 
define $x_{v-2} = \alpha_1$ and then interchange $x_{v-2}$ with $x_{\kappa}$.

Now, when we initially choose $x_1, x_2, x_3, \dots$, we have no idea which value $\alpha_1$ we will 
be interchanging with $x_{\kappa}$. So it is necessary to ensure that any value we ``swap
in'' will not result in a block being contained in four successive points of the sequencing.
Clearly we only have to worry about the first $8+3=11$ points, $x_1, x_2, x_3, \dots, x_{11}$.

Define \[Y = \big\{ \third(x_i,x_j): 1 \leq i < j \leq 11, |i-j| \leq 3 \big\} 
\setminus \{x_1, \dots , x_{11}\}.\]
(Note, in the definition of $Y$, that we do not care about pairs of points that
are more than three positions apart.)
Denote the points in $Y$ as $y_1, \dots , y_m$. It is not hard to verify that
$m \leq 27$, because there are ten pairs $x_i,x_j$ in $\{x_1,\dots , x_{11}\}$ with $j-i = 1$,
nine pairs with $j-i = 2$ and eight pairs with $j-i=3$.

Having already chosen
$x_1, \dots, x_{11}$, we want to ``pre-specify'' some of the next points (this will require
a small modification to the greedy algorithm). To be specific, we define
$x_{14} = y_1$, $x_{16} = y_2$, $\dots$, $x_{2m+12} = y_m$. 
Note that no three of the $y_i$'s are contained
in four consecutive points of the sequencing, from $x_{12}$ to $x_{2m+12}$.

The following diagram might be helpful in the subsequent discussion:
\begin{center}
\fb{\RED{$x_1$}}\fb{\RED{$x_2$}}\fb{\RED{$x_3$}}\fb{\RED{$x_4$}}\fb{\RED{$x_5$}}\fb{\RED{$x_6$}}\fb{\RED{$x_7$}}\fb{\RED{$x_8$}}%
\fb{\RED{$x_9$}}\fb{\RED{$x_{10}$}}\fb{\RED{$x_{11}$}}%
\fb{$x_{12}$}\fb{$x_{13}$}\vspace{.25in}\\
\fb{\RED{$y_1$}}\fb{$x_{15}$}\fb{\RED{$y_2$}}\fb{$x_{17}$}
\,
\raisebox{1ex}{$\cdots$} \,  \fbb{$x_{2m+7}$}\fbb{\RED{$y_{m-2}$}}\fbb{$x_{2m+9}$}\fbb{\RED{$y_{m-1}$}}\fbb{$x_{2m+11}$}\fbb{\RED{$y_m$}}
\end{center}
In this diagram, the red values have been defined and we need to determine the black values.
Let's consider how the greedy algorithm must be modified in order to accomplish this.
We have the following additional restrictions ``looking ahead'' when choosing values
for $x_{12}, x_{13}, x_{15}, \dots ,x_{2m+11}$:
\begin{itemize}
\item each of $x_{12}, x_{13}, x_{15}, \dots ,x_{2m+11}$ must be distinct from
$y_1, \dots , y_m$;
\item we require  that  $\{x_{11},x_{12},y_1\} \not\in \B$, so we must define 
\[x_{12} \neq \third(x_{11},y_1);\]
\item we require  that  \[\{x_{11},x_{13},y_1\}, \{x_{12},x_{13},y_1\}, \{x_{13},y_1,y_2\} \not\in \B,\]
so  we must define \[x_{13} \neq \third(x_{11},y_1), \third(x_{12},y_1), \third(y_1,y_2);\]
\item we require  that  \[\{x_{13},x_{15},y_2\}, \{y_1,x_{15},y_2\} , \{x_{15},y_2,y_3\}\not\in \B,\]
so  we must define \[x_{15} \neq \third(x_{13},y_2), \third(y_1,y_2),  \third(y_2,y_3);\]
\item $\dots$
\item we require  that  
\[\{x_{2m+7},x_{2m+9},y_{m-1}\}, \{y_{m-2},x_{2m+9},y_{m-1}\}, 
\{x_{2m+9},y_{m-1},y_m\} \not\in \B,\]
so  we must define \[x_{2m+9} \neq \third(x_{2m+7},y_{m-1}), \third(y_{m-2},y_{m-1}),  \third(y_{m-1},y_m); \]
\item we require  that   \[\{x_{2m+9},x_{2m+11},y_m\}, \{y_{m-1},x_{2m+11},y_m\} \not\in \B,\]
so  we must define \[x_{2m+11} \neq \third(x_{2m+9},y_{m}), \third(y_{m-1},y_m).\]
\end{itemize}
Of course, we need to ensure that a greedy algorithm can choose values for all these $x_i$'s. 

Now consider what happens when we swap $x_{\kappa}$ with $\alpha_1$. The value $\alpha_1 \not\in Y$, 
so $\alpha_1$ cannot form a block with any two of the points $x_1, \dots, x_{11}$. Since $\kappa \leq 8$,
there are no blocks contained in any four consecutive points chosen from the first
11 points of the sequencing. At the opposite end, we have guaranteed that there are
no blocks contained in any four consecutive points chosen from the last
six points of the sequencing, because of the way that $x_{\kappa}$ was chosen.

The resulting algorithm has the high-level structure described in Figure \ref{alg2}.
\begin{figure}[htb]
\begin{enumerate}
\item Determine $x_1, \dots , x_{11}$ using the greedy approach.
\item Fill in the values $y_1, \dots , y_m$ and the determine the remaining
values $x_{12}, \dots , x_{2m+11}$ using the ``modified'' greedy approach.
\item Determine $x_{2m+13}, \dots , x_{v-3}$ using the greedy approach.
\item Define the values $x_{v-2}= \alpha_1, x_{v-1}= \alpha_2, x_{v}= \alpha_3$ as described in the text,
and then swap $x_{v-2}$ with $x_{\kappa}$.
\item {\bf Return} $(\pi = [x_1\; x_2 \;  \cdots \; x_v])$.
\end{enumerate}
\caption{Algorithm to find a $4$-good sequencing for an STS$(v)$, $(X,\B)$}
\label{alg2}
\end{figure}

All the above steps can be carried out if we ensure that the first $2m+12$ elements of the sequencing
do not overlap with the last six elements of the sequencing. Since $m \leq 27$, this  condition is
guaranteed to hold if $v - 5 \geq 2 \times 27 + 12 + 1$, or $v \geq 72$. So we have proven the following.

\begin{theorem}
\label{T3}
Suppose $v > 71$ and $(X, \B)$ is an STS$(v)$ on points $X = \{1, \dots , v\}$.
Then there is a sequencing $\pi = [x_1\; x_2 \;  \cdots \;  x_v]$ of $X$ that is 4-good
for $(X, \B)$.
\end{theorem}

A similar result can also be proven for PSTS$(v)$ using this technique.

\section{Conclusion}

We make the following conjecture:
For any integer $\ell \geq 3$, there is an integer $n(\ell)$ such that any 
STS$(v)$  with $v \geq n(\ell)$ has an $\ell$-good sequencing.

\end{document}